\documentclass[12pt]{article}
\usepackage{epsf}

\usepackage{amsthm}
\usepackage{amssymb}
\usepackage{amsmath}
\usepackage{amscd}
\usepackage[top=1.5cm,bottom=1.5cm,left=2.5cm,right=2.5cm,includehead,includefoot
           ]{geometry}

\makeatletter

\numberwithin{equation}{section}
\newtheorem{theorem}{Theorem}[section]
\newtheorem{lemma}[theorem]{Lemma}
\newtheorem{corollary}[theorem]{Corollary}

\title{Characterizations of Jordan mappings on some rings and algebras through zero products}
\author{\begin{tabular}{c} Wenbo Huang$,$ Jiankui Li\footnote{Corresponding author.
E-mail address: jiankuili@yahoo.com}~   and Jun He
\\{\small\it Department of Mathematics, East China University of
Science and Technology}\\
{\small\it Shanghai 200237, China}
\end{tabular}}
\date{}
\begin{document}
\maketitle \abstract Let  $\mathcal{U}=\left[
                   \begin{array}{cc}
                     \mathcal{A} & \mathcal{M}\\
                     \mathcal{N}& \mathcal{B} \\
                   \end{array}
                 \right]$  be a  generalized matrix ring, where  $\mathcal{A}$ and
$\mathcal{B}$ are  2-torsion free.  We prove that if  $\phi :\mathcal{U}\rightarrow \mathcal{U}$  is an  additive
mapping such that  $\phi(U)\circ V+U\circ \phi(V)=0$  whenever $UV=VU=0,$   then $\phi=\delta+\eta$, where $\delta$ is
a Jordan derivation and $\eta$ is a multiplier. As its applications, we  prove that the similar conclusion remains valid on
 full matrix algebras,  unital prime rings with a nontrivial idempotent,  unital standard operator algebras, CDCSL algebras and von Neumann algebras.


\

\
\section{Introduction}\

Let $\mathcal{R}$ be a ring (or algebra) with the unit $I$. For each $A,B$ in $\mathcal{R},$ We define the \emph{Jordan product} by  $A\circ B=AB+BA,$ and the  \emph{Lie product} by $[A,B]=AB-BA.$

$\mathcal{R}$ is said to be \emph{2-torsion free}
if $2A=0$ implies $A=0$ for every $A$ in $\mathcal{R}$. Recall that a ring $\mathcal{R}$
is \emph{prime} if $A\mathcal{R}B={\{0\}}$ implies that either $A=0$ or $B=0$ for each $A,B$ in $\mathcal{R}$,
and is \emph{semiprime} if $A\mathcal{R}A={\{0}\}$ implies $A=0$ for every $A$ in $\mathcal{R}$.

Let $\mathcal{M}$ be an $\mathcal{R}$-bimodule. An
additive mapping (or linear mapping) $\eta$ from $\mathcal{R}$ into $\mathcal{M}$ is
called a \emph{multiplier} if $\eta(A)=\eta(I)A=A\eta(I)$ for every
$A$ in $\mathcal{R}$.  An additive mapping (or linear mapping)  $\delta$ from
$\mathcal{R}$ into $\mathcal{M}$ is called a \emph{derivation} if
$\delta(AB)=\delta(A)B+A\delta(B)$ for each $A,B$ in $\mathcal{R}.$  Let $M$ be an element in $\mathcal{M},$ the mapping $\delta_{M}:\mathcal{R}\rightarrow \mathcal{M}, ~A\rightarrow \delta_{M}(A):=[M,A],$ is a derivation. A derivation $\delta:\mathcal{R}\rightarrow \mathcal{M}$
 is said to be an \emph{inner derivation} when it can be written in the form $\delta=\delta_{M}$ for some $M$ in $\mathcal{M}.$  An fundamental contribution, due to Sakai, states that
every derivation on a von Neumann algebra is inner(cf. \cite{sakai}).

 An additive mapping (or linear mapping)  $\delta$ from
$\mathcal{R}$ into $\mathcal{M}$ is called a \emph{Jordan derivation} if $\delta(A^{2})
=\delta(A)A+A\delta(A)$ for every
$A$ in $\mathcal{R}$.
Obviously, every derivation is a Jordan
derivation, while the converse is in general not true. A classical result
of  Herstein \cite{I. Herstein} asserts that every Jordan
derivation on a 2-torsion free prime ring is a derivation. In
\cite{Cusack}, Cusack generalizes the above result
to 2-torsion free semiprime rings. In \cite{444}, Zhang and Yu show that every Jordan derivation of triangular algebras is a derivation.
In \cite{448}, Alizadeh shows that every Jordan derivation from full matrix algebras $M_{n}(\mathcal{A})~(n\geq2)$ into $M_{n}(\mathcal{M})$ is a derivation, where $\mathcal{A}$ is a unital ring and $\mathcal{M}$ is a 2-torsion free $\mathcal{A}$-bimodule. In \cite{222}, Peralta and Russo show that every Jordan derivation from a C*-algebra $\mathcal{A}$ to a Banach $\mathcal{A}$-bimodule is a derivation.

Let~$\mathcal{A}$~and~$\mathcal{B}$~be two unital rings, and $\mathcal{M}$ be a unital
$(\mathcal{A},\mathcal{B})$-bimodule.
We say that $\mathcal{M}$  a \emph{left faithful} $\mathcal{A}$-module
if $A\mathcal{M}=\{0\}$ implies that $A=0$ for every $A\in\mathcal{A}$,
and $\mathcal{M}$  a \emph{right faithful} $\mathcal{B}$-module
if $\mathcal{M}B=\{0\}$ implies that $B=0$ for every $B\in\mathcal{B}$.
If $\mathcal{M}$ is both a left faithful $\mathcal{A}$-module and a
right faithful $\mathcal{B}$-module, then we call $\mathcal{M}$ is a untial  faithful
$(\mathcal{A},\mathcal{B})$-bimodule.

A \emph{Morita context} is a set
$(\mathcal{A},\mathcal{B},\mathcal{M},\mathcal{N})$  and two
mappings $\sigma$ and $\rho$,  where $\mathcal{A}$ and $\mathcal{B}$
are two unital rings,
$\mathcal{M}$ is a untial faithful  $(\mathcal{A},\mathcal{B})$-bimodule
and $\mathcal{N}$ (not necessarily faithful) is a
$(\mathcal{B},\mathcal{A})$-bimodule, $\sigma:
\mathcal{M}\bigotimes_{\mathcal{B}}\mathcal{N}\rightarrow\mathcal{A}$
and
$\rho:\mathcal{N}\bigotimes_{\mathcal{A}}\mathcal{M}\rightarrow\mathcal{B}$
are two homomorphisms satisfying the following commutative
diagrams:
$$
\CD
   \mathcal{M}\otimes_{\mathcal{B}}\mathcal{N}\otimes_{\mathcal{A}}\mathcal{M} @>\sigma\otimes I_{\mathcal{M}}>> \mathcal{A}\otimes_{\mathcal{A}}\mathcal{M} \\
   @V I_{\mathcal{M}}\otimes\sigma VV @V \cong VV  \\
    \mathcal{M}\otimes_{\mathcal{A}}\mathcal{B}@>\cong>> \mathcal{M},
\endCD
$$
and

$$
\CD
   \mathcal{N}\otimes_{\mathcal{A}}\mathcal{M}\otimes_{\mathcal{B}}\mathcal{N} @>\rho\otimes I_{\mathcal{N}}>> \mathcal{B}\otimes_{\mathcal{B}}\mathcal{N} \\
   @V I_{\mathcal{N}}\otimes\rho VV @V \cong VV  \\
    \mathcal{N}\otimes_{\mathcal{B}}\mathcal{A}@>\cong>> \mathcal{N}.
\endCD
$$
These conditions insure that the set

$$\mathcal{U}=\left[
                   \begin{array}{cc}
                   \mathcal{A}& \mathcal{M}\\
                     \mathcal{N}& \mathcal{B} \\
                   \end{array}
                 \right]=\left\{\left(
                   \begin{array}{cc}
                     A & M \\
                     N& B \\
                   \end{array}
                 \right):A\in\mathcal{A},M\in\mathcal{M}, N\in\mathcal{N}, B\in\mathcal{B}\right\}$$
forms a ring  under usual matrix addition and matrix
multiplication if we put $MN=\sigma(M,N)$  and $NM=\rho(N,M)$. We
call it   a \emph{generalized matrix ring.}

Let $\mathcal{A}$ and $\mathcal{B}$ be two unital   rings,
and $\mathcal{M}$ be a
unital faithful $(\mathcal{A},\mathcal{B})$-bimodule.
The set
$$\mathcal{T}=\left[
                   \begin{array}{cc}
                   \mathcal{A}& \mathcal{M}\\
                     0 & \mathcal{B} \\
                   \end{array}
                 \right]=\left\{\left(
                   \begin{array}{cc}
                     A & M \\
                     0 & B \\
                   \end{array}
                 \right):A\in\mathcal{A}, B\in\mathcal{B}, M\in\mathcal{M}\right\}$$
under the usual matrix addition and matrix multiplication is called
a \emph{triangular ring.}

In \cite{521,501,523,446,448,999,888,599}, several authors consider the following conditions on an additive (or a linear) mapping $\phi:\mathcal{R}\rightarrow \mathcal{M}:$
\begin{align*}
 ~~~~~~~~~~~~~~~~~&A,B,C \in\mathcal{R},~~~~AB=BC=0~\Longrightarrow A\phi(B)C=0;~~~~~~~~~~~~~~~~~~~~~~~(\mathbb{P}_{1})\notag\\
&A,B \in\mathcal{R},~~~~~~~AB=0~\Longrightarrow \phi(A)B+A\phi(B)=0;~~~~~~~~~~~~~~~~~~~~~~(\mathbb{P}_{2})\notag\\
&A,B \in\mathcal{R},~~~~~~~AB=BA=0~\Longrightarrow \phi(A)B+A\phi(B)=0;~~~~~~~~~~~~~~(\mathbb{P}_{3})\notag\\
&A,B \in\mathcal{R},~~~~~~~A\circ B=0~\Longrightarrow \phi(A)\circ B+A\circ\phi(B)=0;~~~~~~~~~~~~~~(\mathbb{P}_{1})\notag\\
&A,B \in\mathcal{R},~~~~~~AB=BA=0 \Longrightarrow\phi(A)\circ B+A\circ \phi(B)=0;~~~~~~~~~~~(\mathbb{P})\notag\
\end{align*}
and investigate whether these conditions characterize derivations or Jordan derivations.

In \cite{888}, Hadwin and Li show that if $\phi$ is a bounded linear mapping  from a unital C*-algebra $\mathcal{A}$
into a Banach $\mathcal{A}$-bimodule $\mathcal{M}$ with $\mathcal{M}^{*}$ weakly sequentially complete satisfying the condition $(\mathbb{P}_{1})$ and $\phi(I)=0$, then $\phi$ is a derivation.

 In \cite{999}, An and Hou show that if  $\phi$ is an additive mapping  from a unital ring  $\mathcal{R}$ with a nontrivial idempotent  into itself satisfying the condition $(\mathbb{P}_{2})$, then $\phi(R)=\delta(R)+\xi R,$ where  $\delta$ is a derivation on  $\mathcal{R}$  and $\xi$  is an element in the center of $\mathcal{R}$.

Let $\mathcal{A}$ be a Banach algebra, and  $\mathcal{M}$ be a Banach $\mathcal{A}$-bimodule. Under some mild conditions on $\mathcal{A},$ in \cite{523},  Alaminos et al. characterize a bounded linear mapping
$\phi$ from $\mathcal{A}$ into $\mathcal{M}$  satisfying the conditions $(\mathbb{P}_{1})$, $(\mathbb{P}_{2})$, $(\mathbb{P}_{3})$ and $(\mathbb{P}_{4})$, respectively.

Obviously, the condition $(\mathbb{P})$ is more general than the conditions $(\mathbb{P}_{1})$, $(\mathbb{P}_{2})$, $(\mathbb{P}_{3})$ and $(\mathbb{P}_{4})$.
In \cite{501},  Alaminos et al.  show that if $\phi$ is a bounded linear mapping  from a C*-algebra $\mathcal{A}$
into an essential Banach $\mathcal{A}$-bimodule $\mathcal{M}$ satisfying the condition $(\mathbb{P}),$  then $\phi=\delta+\eta,$ where $\delta$ is a derivation and $\eta$ is a multiplier.
In \cite{599}, Liu and Zhang show that if $\phi$ is a linear mapping  from a triangular algebra $\mathcal{A}$ into itself satisfying the condition $(\mathbb{P}),$ then $\phi=\delta+\eta,$ where $\delta$ is a derivation and $\eta$ is a multiplier.

This paper is organized as follows.
In Section 2, our purpose is to characterize an additive mapping $\phi:\mathcal{U} \rightarrow \mathcal{U}$  satisfying the condition $(\mathbb{P}),$ where $\mathcal{U}$
is a generalized matrix ring with $\mathcal{A}$ and $\mathcal{B}$ are 2-torsion free.
We prove that  $\phi=\delta+\eta$, where $\delta$ is a Jordan derivation from
$\mathcal{U}$ into itself and $\eta$ is a multiplier from $\mathcal{U}$ into itself. As its applications, we  prove that the similar conclusion remains valid on
 full matrix algebras,  unital prime rings with a nontrivial idempotent,  unital standard operator algebras, CDCSL algebras and von Neumann algebras. We extend several  results in \cite{999,599}.

\section{The Main Results}
$$ $$

In the section, we assume that  $~\mathcal{U}=\left[
                   \begin{array}{cc}
                     \mathcal{A} & \mathcal{M}\\
                     \mathcal{N}& \mathcal{B} \\
                   \end{array}
                 \right]$  is a generalized matrix ring,  where  $\mathcal{A}$ and $\mathcal{B}$
are 2-torsion free.

 We denote by
$\mathbf{\mathrm{Z}(\mathcal{U})}$ be the center of $~\mathcal{U} $.
 It is well known that
$$Z(\mathcal{U})=\left\{\left(
                   \begin{array}{cc}
                     A & 0 \\
                     0 & B \\
                   \end{array}
                 \right):AM=MB, NA=BN, M\in\mathcal{M} , N\in\mathcal{N}\right\}.$$
 Let $I _{\mathcal{A}}$  and $I _{\mathcal{B}}$ be the identities of the
rings $\mathcal{A}$ and $\mathcal{B}$, respectively.  $I$ denotes the
identity of $~\mathcal{U}.$

In the following, let
$$P_{1} = \left(\begin{array}{cc}
 I_{\mathcal{A}}&0\\
0&0 \\
\end{array}\right), P_{2} = \left(\begin{array}{cc}
 0&0\\
0&I_{\mathcal{B}}\\
\end{array}\right)~~  and~
 ~\mathcal{U}_{ij} = P_{i}\mathcal{U}P_{j}    ~~( i,j=1,2).$$
It is clear that the generalized matrix ring may be represented as
$~\mathcal{U}=\mathcal{U}_{11}+\mathcal{U}_{12}+\mathcal{U}_{21}+\mathcal{U}_{22}.$
 Here $\mathcal{U}_{11}$ and $\mathcal{U}_{22}$ are subrings of
$\mathcal{U}$ isomorphic to $\mathcal{A}$ and $\mathcal{B}$ ,
respectively.  $\mathcal{U}_{12}$ is a
$(\mathcal{U}_{11},\mathcal{U}_{22})$-bimodule isomorphic to the
$(\mathcal{A},\mathcal{B})$-bimodule $\mathcal{M}$, and
$\mathcal{U}_{21}$ is a
$(\mathcal{U}_{22},\mathcal{U}_{11})$-bimodule isomorphic to the
$(\mathcal{B},\mathcal{A})$-bimodule $\mathcal{N}.$

 The following theorem is the main result of the paper.
\begin{theorem}
Suppose that ~ $\mathcal{U}=\left[
                   \begin{array}{cc}
                     \mathcal{A} & \mathcal{M}\\
                     \mathcal{N}& \mathcal{B} \\
                   \end{array}
                 \right]$  is a  generalized matrix ring, where  $\mathcal{A}$ and
$\mathcal{B}$ are 2-torsion free.  If  $\phi$ : $\mathcal{U}\rightarrow \mathcal{U}$  is an additive mapping satisfying the
condition $(\mathbb{P}),$ then there exist a Jordan derivation $\delta$ from
$\mathcal{U}$ into
 itself and a multiplier $\eta$ from $\mathcal{U}$ into
 itself  such that $\phi=\delta+\eta$. In addition, if
 $\phi(I)=0$,  then $\phi$ is a Jordan derivation.

\end{theorem}

To show the result, we need  the following several lemmas.
\begin{lemma}\
$\phi(I)\in~Z(\mathcal{U})$.
\end{lemma}

\begin{proof}
Let $P$ be an arbitrary idempotent in $\mathcal{U}$, and
$P(I-P)=(I-P)P=0.$
By assumption, it follows that
$$\phi(P)\circ(I-P)+P\circ\phi(I-P)=0.$$
Hence
$$2\phi(P)+P\phi(I)+\phi(I)P=2\phi(P)P+2P\phi(P).$$
Multiplying the above equality on the left and right by $P$
respectively, we have that
$$P\phi(I)+P\phi(I)P=2P\phi(P)P$$
and $$\phi(I)P+P\phi(I)P=2P\phi(P)P.$$
Thus
\begin{align}
 P\phi(I)=\phi(I)P.\label{202}
\end{align}
In particular,
\begin{align}
\phi(I)=P_{1}\phi(I)P_{1}+P_{2}\phi(I)P_{2}.\label{203}
\end{align}
Let  $U_{ij}$ be any elements in $\mathcal{U}_{ij}$  $(
i,j=1,2),$ respectively.
Since $P_{1}+U_{12}$ and $P_{2}+U_{21}$ are idempotents in
$~\mathcal{U},$  we have
$$(P_{1}+U_{12})\phi(I)=\phi(I)(P_{1}+U_{12})$$
and $$(P_{2}+U_{21})\phi(I)=\phi(I)(P_{2}+U_{21}).$$
By \eqref{202},  we have that
\begin{align}
U_{12}\phi(I)=\phi(I)U_{12}\label{299}
\end{align}
and
\begin{align}
U_{21}\phi(I)=\phi(I)U_{21}.\label{288}
\end{align}
By \eqref{203}, \eqref{299} and \eqref{288},  it follows that
$$P_{1}\phi(I)P_{1}U_{12}=U_{12}P_{2}\phi(I)P_{2}.$$
Similarly,
$$P_{2}\phi(I)P_{2}U_{21}=U_{21}P_{1}\phi(I)P_{1}.$$
Hence  $$ \phi(I)\in Z(\mathcal{U}).$$
The proof is complete.
\end{proof}

\begin{lemma}
$P_{1}\phi(P_{2})P_{1}=P_{2}\phi(P_{1})P_{2}=0.$
\end{lemma}
\begin{proof}
By  \eqref{203},  we have that
$$ P_{1}\phi(I)P_{1}=P_{1}\phi(P_{1})P_{1}$$
and $$P_{2}\phi(I)P_{2}=P_{2}\phi(P_{2})P_{2}.$$
Hence
$$P_{1}\phi(P_{2})P_{1}=P_{2}\phi(P_{1})P_{2}=0.$$
The proof is complete.\end{proof}

For each $U$ in $\mathcal{U},$ in the following, we define an additive mapping  $\varphi:  \mathcal{U} \rightarrow \mathcal{U}~$  by
$$\varphi(U)=\phi(U)-\phi(I)U+[P_{1}\phi(P_{1})P_{2}+P_{2}\phi(P_{2})P_{1},U].$$
\begin{lemma}
$\varphi$ has the following properties:

$(1)$  ~~   $\varphi(P_{1})=\varphi(P_{2})=0;$

$(2)$   ~~$\varphi$  satisfying the
condition $(\mathbb{P}).$
\end{lemma}
\begin{proof}
By Lemma 2.3, it follows that
\begin{align*}
\varphi(P_{1})&=\phi(P_{1})-\phi(I)P_{1}+[P_{1}\phi(P_{1})P_{2}+P_{2}\phi(P_{2})P_{1},P_{1}]\notag\\
&=\phi(P_{1})-\phi(I)P_{1}+P_{2}\phi(P_{2})P_{1}-P_{1}\phi(P_{1})P_{2}\notag\\
&=\phi(P_{1})-\phi(P_{1})P_{1}-\phi(P_{2})P_{1}+P_{2}\phi(P_{2})P_{1}-P_{1}\phi(P_{1})P_{2}\notag\\
&=\phi(P_{1})P_{2}-\phi(P_{2})P_{1}+P_{2}\phi(P_{2})P_{1}-P_{1}\phi(P_{1})P_{2}\notag\\
&=0.\notag
\end{align*}
Similarly, we can show $\varphi(P_{2})=0.$ Hence $(1)$ holds.

Let $UV=VU=0,$  by the assumption and Lemma 2.2, it follows that
\begin{align*}
\varphi(U)\circ V&+U\circ \varphi(V)\notag\\
&=\phi(U)\circ V+U\circ \phi(V)+[T,U]\circ V
+U\circ[T,V]\notag\\
&=0+[T,U]\circ V+U\circ[T,V]\notag\\
&=0,\notag
\end{align*}
where $T=P_{1}\phi(P_{1})P_{2}+P_{2}\phi(P_{2})P_{1}.$
Hence $(2)$ holds.
\end{proof}

\begin{lemma}\
 $\varphi$ has the following properties:

 $({1})$  ~~    $ \varphi(U_{11}) \in \mathcal{U}_{11};$

$({2})$  ~~    $  \varphi(U_{22}) \in \mathcal{U}_{22};$

$({3})$ ~~
$P_{1}\varphi(U_{12})P_{1}=P_{2}\varphi(U_{12})P_{2}=0 ;$

$({4})$ ~~
$P_{1}\varphi(U_{21})P_{1}=P_{2}\varphi(U_{21})P_{2}=0 .$

\end{lemma}

\begin{proof}
Since  $U_{11}P_{2}=P_{2}U_{11}=0,$ by Lemma 2.4, we have that
\begin{align}
P_{2}\varphi(U_{11})+\varphi(U_{11})P_{2}=0. \label{204}
\end{align}
Multiplying \eqref{204} on the left by $P_{1}$, we have that
$$P_{1}\varphi(U_{11})P_{2}=0.$$
Multiplying \eqref{204} on the right by $P_{1}$, we have that
$$P_{2}\varphi(U_{11})P_{1}=0.$$
Multiplying \eqref{204} on the left and  right by $P_{2}$
respectively,
we have that
$$P_{2}\varphi(U_{11})P_{2}=0.$$
Hence we obtain  $$\varphi(U_{11})\in \mathcal{U}_{11}.$$
Similarly,   $$\varphi(U_{22}) \in \mathcal{U}_{22}.$$
Since  $U_{21}U_{21}=U_{21}U_{21}=0$ , we have
 $$2\varphi(U_{21})U_{21}+2U_{21}\varphi(U_{21})=0,$$
 and$$(P_{2}+U_{21})(P_{1}-U_{21})=(P_{1}-U_{21})(P_{2}+U_{21})=0.$$
We have that
\begin{align}P_{1}\varphi(U_{21})+\varphi(U_{21})P_{1}=P_{2}\varphi(U_{21})+\varphi(U_{21})P_{2}.\label{205}
\end{align}
Multiplying \eqref{205} on the left and right by $P_{1}$
respectively,
we have that  $$P_{1}\varphi(U_{21})P_{1}=0.$$
Multiplying \eqref{205} on the left and  right by $P_{2}$
respectively, we have
$$P_{2}\varphi(U_{21})P_{2}=0.$$
Similarly, $$P_{1}\varphi(U_{12})P_{1}=P_{2}\varphi(U_{12})P_{2}=0.$$
The proof is complete.
\end{proof}
\begin{lemma}
Let $U_{ij}$ and $V_{ij}$ be arbitrary elements in $\mathcal{U}_{ij}$
(i,j=1,2), respectively.
$\varphi$   has the following properties:

 $({1})$ ~~   $ \varphi (U_{11}\circ V_{12})= \varphi(U_{11})\circ V_{12}+U_{11}\circ \varphi(V_{12});$

 $({2})$ ~~   $ \varphi (U_{22}\circ V_{21})= \varphi(U_{22})\circ V_{21}+U_{22}\circ \varphi(V_{21});$

 $({3})$ ~~   $ \varphi (U_{11}\circ V_{21})= \varphi(U_{11})\circ V_{21}+U_{11}\circ \varphi(V_{21});$

$({4})$ ~~  $ \varphi (U_{22}\circ V_{12})=
 \varphi(U_{22})\circ V_{12}+U_{22}\circ \varphi(V_{12});$

 $({5})$ ~~   $ \varphi (U_{11}\circ V_{11})= \varphi(U_{11})\circ V_{11}+U_{11}\circ \varphi(V_{11});$

 $({6})$ ~~   $ \varphi (U_{22}\circ V_{22})= \varphi(U_{22})\circ V_{22}+U_{22}\circ \varphi(V_{22});$

 $({7})$ ~~   $ \varphi (U_{12}\circ V_{21})= \varphi(U_{12})\circ V_{21}+U_{12}\circ \varphi(V_{21});$

 $({8})$ ~~   $ \varphi (U_{21}\circ V_{12})= \varphi(U_{21})\circ V_{12}+U_{21}\circ \varphi(V_{12}).$
\end{lemma}

\begin{proof}
Since$$(U_{11}-U_{11}V_{12})(V_{12}+P_{2})=(V_{12}+P_{2})(U_{11}-U_{11}V_{12})=0,$$
by Lemma 2.4, it follows that
$$\varphi(U_{11}-U_{11}V_{12})\circ(V_{12}+P_{2})+(U_{11}-U_{11}V_{12})\circ
\varphi(V_{12}+P_{2})=0.$$
We have that
\begin{align*}
(-U_{11}V_{12}\varphi(V_{12})&-\varphi(V_{12})U_{11}V_{12}-V_{12}\varphi(U_{11}V_{12}))+(\varphi(U_{11})V_{12}-\varphi(U_{11}V_{12})P_{2}\notag\\
&+U_{11}\varphi(V_{12}))+(-P_{2}\varphi(U_{11}V_{12})+\varphi(V_{12})U_{11})+(-\varphi(U_{11}V_{12})V_{12})\notag\\
&=0.\notag
\end{align*}
By Lemma 2.5, we have that
$$\varphi(U_{11})V_{12}-\varphi(U_{11}V_{12})P_{2}+U_{11}\varphi(V_{12})\in \mathcal{U}_{12},$$
and
$$-P_{2}\varphi(U_{11}V_{12})+\varphi(V_{12})U_{11}\in\mathcal{U}_{21}.$$
Thus
\begin{align}
\varphi(U_{11})V_{12}-\varphi(U_{11}V_{12})P_{2}+U_{11}\varphi(V_{12})-P_{2}\varphi(U_{11}V_{12})+\varphi(V_{12})U_{11}=0.\label{456}
\end{align}
By Lemma 2.5, we have that
\begin{align}
\varphi(U_{11}V_{12})P_{2}+P_{2}\varphi(U_{11}V_{12})=\varphi(U_{11}V_{12}).\label{206}
\end{align}
By Lemma 2.5,  \eqref{206} and \eqref{456}, we obtain
\begin{align}
\varphi(U_{11}\circ V_{12})=\varphi(U_{11}) \circ V_{12}+U_{11}\circ\varphi(V_{12}).\label{458}
\end{align}
Hence $({1})$  holds.

Since
$$(V_{21}-P_{2})(V_{21}U_{11}+U_{11})=(V_{21}U_{11}+U_{11})(V_{21}-P_{2})=0,$$
by Lemma 2.4, it follows that
$$\varphi(V_{21}-P_{2}) \circ(V_{21}U_{11}+U_{11})+(V_{21}-P_{2})
\circ \varphi(V_{21}U_{11}+U_{11})=0.$$
Hence
\begin{align*}
(\varphi(V_{21})V_{21}U_{11}&+\varphi(V_{21}U_{11})V_{21})+(U_{11}\varphi(V_{21})-\varphi(V_{21}U_{11})P_{2})\notag\\
&+(\varphi(V_{21})U_{11}-P_{2}\varphi(V_{21}U_{11})+V_{21}\varphi(U_{11}))\notag\\
&+(V_{21}U_{11}\varphi(V_{21})+V_{21}\varphi(V_{21}U_{11}))=0.\notag
\end{align*}
By Lemma 2.5, we have that
$$U_{11}\varphi(V_{21})-\varphi(V_{21}U_{11})P_{2} \in\mathcal{U}_{12},$$
and
$$\varphi(V_{21})U_{11}-P_{2}\varphi(V_{21}U_{11})+V_{21}\varphi(U_{11})
\in \mathcal{U}_{21},$$
it follows that
\begin{align}
U_{11}\varphi(V_{21})-\varphi(V_{21}U_{11})P_{2}+\varphi(V_{21})U_{11}-P_{2}\varphi(V_{21}U_{11})+V_{21}\varphi(U_{11})=0.\label{457}
\end{align}
By Lemma 2.5, we have that
 \begin{align}
\varphi(V_{21}U_{11})P_{2}+P_{2}\varphi(V_{21}U_{11})=\varphi(V_{21}U_{11}).\label{207}
\end{align}
By Lemma 2.5,  \eqref{457} and \eqref{207},  we obtain
$$\varphi(U_{11}\circ V_{21})=\varphi(U_{11}) \circ
V_{21}+U_{11}\circ\varphi(V_{21}).$$
Hence $({3})$  holds.

Let  $V_{12}$  be an arbitrary element in
$\mathcal{U}_{12}.$  Since
$$\varphi(U_{11}V_{11}V_{12})=\varphi(U_{11}V_{11} \circ V_{12}),$$
by \eqref{458},  we have that
\begin{align}
\varphi(U_{11}V_{11}V_{12})=\varphi(U_{11}V_{11})V_{12}+U_{11}V_{11}\varphi(V_{12})+\varphi(V_{12})U_{11}V_{11},\label{208}
\end{align}
on the other hand, $$\varphi(U_{11}V_{11}V_{12})=\varphi(U_{11}\circ
V_{11}V_{12}).$$
By Lemma 2.5 and \eqref{458} it follows that
\begin{align}
\varphi(U_{11}V_{11}V_{12})=\varphi(U_{11})V_{11}V_{12}+U_{11}\varphi(V_{11})V_{12}+U_{11}V_{11}\varphi(V_{12})+\varphi(V_{12})V_{11}U_{11}.\label{209}
\end{align}
By \eqref{208} and \eqref{209}, we have that
$$(\varphi(U_{11}V_{11})-\varphi(U_{11})V_{11}-U_{11}\varphi(V_{11}))V_{12}+\varphi(V_{12})U_{11}V_{11}-\varphi(V_{12})V_{11}U_{11}=0.$$
By Lemma 2.5, we have that
$$(\varphi(U_{11}V_{11})-\varphi(U_{11})V_{11}-U_{11}\varphi(V_{11}))V_{12}=0.$$
By  assumption, $~\mathcal{U}_{12}$ is a left faithful $\mathcal{U}_{11}$-module. It follows that
$$\varphi(U_{11}V_{11})=\varphi(U_{11})V_{11}+U_{11}\varphi(V_{11}).$$
Hence $({5})$  holds.

Since
\begin{align*}
(U_{12}V_{21}&+U_{12}+V_{21}+P_{2})(V_{21}U_{12}-V_{21}-U_{12}+P_{1})\notag\\
&=(V_{21}U_{12}-V_{21}-U_{12}+P_{1})(U_{12}V_{21}+U_{12}+V_{21}+P_{2})\notag\\
&=0.\notag
\end{align*}
By Lemma 2.4, it follows that
\begin{align*}
\varphi(U_{12}V_{21}&+U_{12}+V_{21}+P_{2})
\circ(V_{21}U_{12}-V_{21}-U_{12}+P_{1})\notag\\
&+(U_{12}V_{21}+U_{12}+V_{21}+P_{2}) \circ
\varphi(V_{21}U_{12}-V_{21}-U_{12}+P_{1})\notag\\
&=0.\notag
\end{align*}
It follows that
\begin{align*}
2(\varphi(U_{12}V_{21})&-\varphi(U_{12})V_{21}-U_{12}\varphi(V_{21}))+(-\varphi(U_{12}V_{21})U_{12}\notag\\
&+\varphi(U_{12})V_{21}U_{12}+\varphi(V_{21})V_{21}U_{12}+P_{1}\varphi(U_{12})+P_{1}\varphi(V_{21})\notag\\
&+U_{12}\varphi(V_{21}U_{12})-U_{12}V_{21}\varphi(V_{21})-U_{12}V_{21}\varphi(U_{12})-\varphi(V_{21})P_{2}\notag\\
&-\varphi(U_{12})P_{2})+(\varphi(U_{12})P_{1}+\varphi(V_{21})P_{1}-V_{21}\varphi(U_{12}V_{21})\notag\\
&+V_{21}U_{12}\varphi(U_{12})+V_{21}U_{12}\varphi(V_{21})-P_{2}\varphi(V_{21})+\varphi(V_{21}U_{12})V_{21}\notag\\
&-\varphi(V_{21})U_{12}V_{21}-\varphi(U_{12})U_{12}V_{21})+2(\varphi(V_{21}U_{12})-\varphi(V_{21})U_{12}\notag\\
&-V_{21}\varphi(U_{12}))=0.\notag
\end{align*}
By Lemma 2.5, it follows that
$$0=2(\varphi(U_{12}V_{21})-\varphi(U_{12})V_{21}-U_{12}\varphi(V_{21}))\in  \mathcal{U}_{11},$$
and
$$0=2(\varphi(V_{21}U_{12})-\varphi(V_{21})U_{12}-V_{21}\varphi(U_{12})\in \mathcal{U}_{22}.$$
By assumption, $\mathcal{U}_{11}$ and $\mathcal{U}_{22}$  are
2-torsion free.  Thus
$$\varphi(U_{12}V_{21})-\varphi(U_{12})V_{21}-U_{12}\varphi(V_{21})=0$$
and
$$\varphi(V_{21}U_{12})-\varphi(V_{21})U_{12}-V_{21}\varphi(U_{12})=0.$$
Hence  $({7})$  holds.

Similarly, we can prove that $({2}),$  $({4}),$ $({6})$
and $({8}).$ The proof is complete.
\end{proof}

\begin{lemma}
$\varphi$ is a Jordan derivation.

\end{lemma}
\begin{proof}
By Lemmas 2.5 and 2.6, the conclusion follows.
\end{proof}

\begin{proof}[\textbf{Proof of Theorem 2.1}]

Let  $U$  be an arbitrary element in $\mathcal{U}.$   We define an additive mapping  $\varphi:  \mathcal{U} \rightarrow \mathcal{U}~$  by
$$\varphi(U)=\phi(U)-\phi(I)U+[P_{1}\phi(P_{1})P_{2}+P_{2}\phi(P_{2})P_{1},U],$$
where
$T=P_{1}\phi(P_{1})P_{2}+P_{2}\phi(P_{2})P_{1}$  in $\mathcal{U}.$
Then
$$\phi(U)=\varphi(U)-[T,U]+\phi(I)U.$$
We define $\delta$ and $\eta$ from $\mathcal{U}$ into itself by
$$\delta(U)=\varphi(U)-[T,U],~~and~~ \eta(U)=\phi(I)U,$$
it follows that
$$\phi=\delta+\eta.$$
By Lemma 2.7, $\delta$ is a Jordan derivation. Hence $\delta(I)=0,$ and $\eta(I)=\phi(I).$
By Lemma 2.2, it follows that $\eta$ is a multiplier. The proof is complete.
\end{proof}

\begin{corollary}
Let $\mathcal{A}$ be a unital ring, and $M_{n}(\mathcal{A}),~n\geq2,$ be the full matrix algebra of all $n\times n$ matrices over $\mathcal{A}.$  If $\phi: M_{n}(\mathcal{A})\rightarrow M_{n}(\mathcal{A})$  is  an additive  mapping  satisfying the condition $(\mathbb{P}),$ then there exist a
derivation $\delta: M_{n}(\mathcal{A})\rightarrow M_{n}(\mathcal{A})$ and a multiplier $\eta:M_{n}(\mathcal{A})\rightarrow M_{n}(\mathcal{A})$  such that
$\phi=\delta+\eta$. In addition, if
 $\phi(I)=0$,  then  $\phi$ is a derivation.
\end{corollary}
\begin{proof}
By Theorem 2.1 and  \cite[Theorem 3.1]{448}, the conclusion follows.
\end{proof}

\begin{corollary}
Let $\mathcal{R}$ be a  2-torsion free unital prime  ring with a
nontrivial idempotent  $P.$  If  $\phi: \mathcal{R}
\rightarrow \mathcal{R}$  be an additive mapping satisfying the  condition $(\mathbb{P})$ , then there
exist a derivation $\delta :\mathcal{R} \rightarrow \mathcal{R}$ and
a multiplier  $\eta :\mathcal{R} \rightarrow \mathcal{R}$  such that
$\phi=\delta+\eta$. In addition, if
 $\phi(I)=0$, then  $\phi$ is a derivation.
\end{corollary}
\begin{proof}
By assumption,  $\mathcal{R}$ is a prime ring, it follows that
$P\mathcal{R}(I-P)$  is a faithful
$(P\mathcal{R}P,(I-P)\mathcal{R}(I-P))$-bimodule. Then
$\mathcal{R}$  is  isomorphic to the generalized matrix ring $$\left[
                                                                 \begin{array}{cc}
                                                                   P\mathcal{R}P & P\mathcal{R}(I-P) \\
                                                                   (I-P)\mathcal{R}P & (I-P)\mathcal{R}(I-P) \\
                                                                 \end{array}
                                                               \right].$$
By Theorem 2.1 and  \cite[Theorem 3.1]{I. Herstein}, the
conclusion follows.
\end{proof}

Let $\mathbf{X}$  be a real or complex Banach space and let
$\mathcal{B}(\mathbf{X})$  and $\mathcal{F}(\mathbf{X})$  denote
the algebra of all bounded linear operators on  $\mathbf{X}$  and
the ideal of all finite rank operators in
$\mathcal{B}(\mathbf{X})$, respectively. A subalgebra
$\mathcal{A}(\mathbf{X})$ of $\mathcal{B}(\mathbf{X})$ is said
to be \emph{standard} in case $\mathcal{F}(\mathbf{X}) \subset
\mathcal{A}(\mathbf{X}).$  Obviously, there
exist nontrivial idempotents in any standard operator algebra. Any standard
operator algebra is prime. We have

\begin{corollary}
Let $\mathbf{X}$  be a real or complex Banach space and let
$\mathcal{A}(\mathbf{X})$  be a unital standard operator algebra on
$\mathbf{X}$. If $\phi: \mathcal{A}(\mathbf{X}) \rightarrow
\mathcal{A}(\mathbf{X})$   is a linear mapping  satisfying the
condition  $(\mathbb{P}),$ then there exist a derivation  $\delta
:\mathcal{A}(\mathbf{X}) \rightarrow \mathcal{A}(\mathbf{X})$ and
a multiplier   $\eta : \mathcal{A}(\mathbf{X}) \rightarrow
\mathcal{A}(\mathbf{X})$ such that  $\phi=\delta+\eta$.  In
addition, if
 $\phi(I)=0$, then $\phi$ is a derivation.

\end{corollary}

\begin{theorem}
Suppose that  $\mathcal{T}=\left[
                   \begin{array}{cc}
                     \mathcal{A} & \mathcal{M}\\
                     0 & \mathcal{B} \\
                   \end{array}
                 \right]$  is  a triangular ring, where  $\mathcal{A}$ and
$\mathcal{B}$ are 2-torsion free.   If  $\phi: \mathcal{T} \rightarrow
\mathcal{T}$  is  an additive  mapping  satisfying the condition $(\mathbb{P}),$  then there exist a
derivation $\delta :\mathcal{T} \rightarrow \mathcal{T}$  and a
multiplier  $\eta :\mathcal{T} \rightarrow \mathcal{T}$  such that
$\phi=\delta+\eta$. In addition, if
 $\phi(I)=0$,  then  $\phi$ is a derivation .

\end{theorem}

\begin{proof}

 The proof is similar to the proof of Theorem 2.1. We leave it to the reader.
\end{proof}

\begin{lemma}

Let  $\mathcal{A}$ be an algebra  over $\mathbb{C}$ with unit $I.$  Suppose that $\mathcal{A}=\sum\limits^{\infty}_{i=1}\bigoplus\mathcal{A}_{i}$, where
$\mathcal{A}_{i}$ is a  unital subalgebra of  $\mathcal{A}$ with unit
$I_{i}~$  for every $i=1,2,\cdots $.  Let $\phi: \mathcal{A} \rightarrow \mathcal{A}$   be
a linear mapping  satisfying the condition  $(\mathbb{P}),$ then
$\phi(\mathcal{A}_{i})\subset \mathcal{A}_{i},~   i=1,2,\cdots $.
\end{lemma}

\begin{proof}

Given any  integer $i$, let $\mathcal{B}=
\sum\limits^{\infty}_{j\neq i}\bigoplus\mathcal{A}_{j}$,
 then
$\mathcal{A}\cong\mathcal{A}_{i}\bigoplus\mathcal{B}$.
By Lemma 2.3, $~\phi(I_{i})  \in \mathcal{A}_{i}$  and $\phi(I-I_{i})
\in
 \mathcal{B}$.
Let $A$  be an arbitrary element in  $\mathcal{A}_{i}$, we have that
$$A(I-I_{i})=(I-I_{i})A=0.$$
By assumption, it follows that
$$2\phi(A)(I-I_{i})+\phi(I-I_{i})A+A\phi(I-I_{i})=0.$$
By assumption, it follows that  $$\phi(A)(I-I_{i})=0.$$
Thus $$\phi(A) \in \mathcal{A}_{i}.$$
The proof is complete.
\end{proof}

\begin{lemma}
Let  $\mathcal{A}$ be an algebra  over $\mathbb{C}$ with unit $I.$   Suppose that $\mathcal{A}=\sum\limits^{\infty}_{i=1}\bigoplus\mathcal{A}_{i}$, where
$\mathcal{A}_{i}$ is a  unital subalgebra of  $\mathcal{A}$ with unit $I_{i}~$  for every $i=1,2,\cdots $.
Let $\phi: \mathcal{A} \rightarrow \mathcal{A}$   be a linear mapping  satisfying the following  conditions:

$({1})$  $\phi(\mathcal{A}_{i}) \subset \mathcal{A}_{i};$

$({2})$  $\phi|_{\mathcal{A}_{i}}=\delta_{i}+\eta_{i},$  where
$\delta_{i}$ is a derivation on $\mathcal{A}_{i}$ and  $\eta_{i}$ is a multiplier on $\mathcal{A}_{i},$  for every $i=1,2\cdots.$

Then there exist a derivation $\delta$ from $\mathcal{A}$  into itself and a
multiplier $\eta$ from  $\mathcal{A}$  into itself, such that $\phi=\delta+\eta.$

\end{lemma}

\begin{proof}
Let $A$ be an arbitrary element in $\mathcal{A}$, then $A$ has the form
$A=\sum\limits^{\infty}_{i=1}\bigoplus A_{i}$, where   $A_{i} \in
\mathcal{A}_{i}$~ for every  $i=1,2,\cdots $.
We define $\delta$ and $\eta$ from  $\mathcal{A}$ into itself by
$$\delta(A)=\sum\limits^{\infty}_{i=1}\bigoplus\delta_{i}(A_{i}),~ and~~\eta(A)=\sum\limits^{\infty}_{i=1}\bigoplus\eta_{i}(A_{i}).$$
It follows that $$\phi(A)=\delta(A)+\eta(A).$$
By assumption,  we may verify that $\delta$ is a derivation and $\eta$ is a multiplier.
\end{proof}
Let $\mathcal{A}$ be a CDCSL (completely distributive commutative
subspace lattice) algebra on a separable Hilbert space $\mathcal H$ (cf. \cite{587} ).
It is well known (see \cite[Lemmas 3 and 4]{518}) that $\mathcal{A}\cong
\sum\limits^{\infty}_{i=1}\bigoplus\mathcal{A}_i$, where each
$\mathcal{A}_i$ is a nest subalgebra of $\mathcal{B}(\mathcal{H}).$

The following theorem generalizes Corollary 2.5 of \cite{999}.
\begin{theorem}
Let  $\mathcal{A}$ be a CDCSL algebra on a separable Hilbert space $\mathcal
 H$.
 If $\phi: \mathcal{A} \rightarrow \mathcal{A}$ is a linear mapping
 satisfying the condition $(\mathbb{P})$,  then there exist a derivation  $\delta
:\mathcal{A} \rightarrow \mathcal{A}$ and a multiplier $\eta :
\mathcal{A} \rightarrow \mathcal{A}$ such that $\phi=\delta+\eta$.
In addition, if
 $\phi(I)=0$, then $\phi$ is a  derivation.

\end{theorem}
\begin{proof}
By Corollary 2.10, Theorem 2.11, Lemmas 2.12 and 2.13, the conclusion follows.
\end{proof}
\begin{lemma}
Let  $\mathcal{A}$ be an abelian von Neumann algebra. If
$\phi: \mathcal{A} \rightarrow \mathcal{A}$  is a linear mapping
 satisfying the  condition $(\mathbb{P}),$  then $\phi$ is a multiplier.
\end{lemma}

\begin{proof}
By assumption, $~\mathcal{A}$ is an abelian von Neumann algebra. It is well known that(\cite[Theorem 5.2.1]{R.Kadison J. Ringrose}) $\mathcal{A}$ is $*$-isomorphic to the algebra $C({\Omega}),$ where $\Omega$ is an extremely disconnected Hausdorff space.
If $A,B \in \mathcal{A}$  with $AB=0$, then the supports of $A$ and $B$ are disadjoint. So
 $$\phi(A)B+A\phi(B)=0$$ implies that $$\phi(A)B=A\phi(B)=0.$$ It follows that,
 $\phi$ is a linear annihilator-preserving mapping (cf.
\cite{580}).
By \cite[Corollary 2.6]{Ahlem Ben Ali Essalen},  $\phi$  is a
multiplier from $\mathcal{A}$ into itself.
\end{proof}

The following theorem generalizes Corollary 2.3 of \cite{999}.
\begin{theorem}
 Let $\mathcal{M}$ be a von Neumann algebra. If
$\phi: \mathcal{M} \rightarrow \mathcal{M}$  is a linear mapping
 satisfying  the  condition $(\mathbb{P})$, then there exist an  inner derivation  $\delta
:\mathcal{M} \rightarrow \mathcal{M}$ and a multiplier $\eta :
\mathcal{M} \rightarrow \mathcal{M}$ such that $\phi=\delta+\eta$.
In addition, if
 $\phi(I)=0$, then $\phi$ is an  inner derivation.

\end{theorem}

\begin{proof}
We denote $\mathcal{M}=\mathcal{A}\bigoplus \mathcal B$, where
$\mathcal{A}=0$ or a von Neumann algebra of type $I_{1}$ and
$\mathcal{B}=0$ or a  von Neumann algebra has no direct summand of type $I_{1}$.

Suppose that $\mathcal{B} \neq 0,$   by Lemma 2.12, $~\phi(\mathcal{B})\subset \mathcal{B}.$
By Lemmas 2.13 and 2.15, we only need to show that
there exist a derivation $\delta_{2}$ on $\mathcal{B}$ and a
multiplier $\eta_{2}$ on $\mathcal{B}$, such that
$\phi|_{\mathcal{B}} =\delta_{2}+\eta_{2}.$

We denote $I_{\mathcal{B}}$ be the unit of $\mathcal{B}.$ For any $A \in
\mathcal{B}$, we denote the central carrier of $A$ by $C_{A}$.
It is well  known
that there exists a projection  $P$   in $\mathcal{B}$  such that
$C_{P}=C_{I_{\mathcal{B}}-P}=I_{\mathcal{B}}$ (cf. \cite[chapters 5 and 6]{R.Kadison J. Ringrose}).  Hence  $P\mathcal{B}(I_{\mathcal{B}}-P)$  is a faithful
$(P\mathcal{B}P,(I_{\mathcal{B}}-P)\mathcal{B}(I_{\mathcal{B}}-P))$ -bimodule.
It is clear that $\mathcal{B}$ is algebraic isomorphic to the
generalized matrix ring

                                                                 $$\left[
                                                                 \begin{array}{cc}
                                                                   P\mathcal{B}P & P\mathcal{B}(I_{\mathcal{B}}-P) \\
                                                                   (I_{\mathcal{B}}-P)\mathcal{B}P & (I_{\mathcal{B}}-P)\mathcal{B}(I_{\mathcal{B}}-P) \\
                                                                 \end{array}
                                                               \right].$$
By Theorem 2.1 and \cite[Corollary 17]{222}, there exist a derivation  $\delta_{2}$ on
$\mathcal{B}$  and a multiplier  $\eta_{2}$ on  $\mathcal{B}$ such
that  $\phi|_{\mathcal{B}} =\delta_{2}+\eta_{2}.$
By Lemmas 2.13, 2.15 and  \cite[Theorem 4.1.6]{sakai}, the conclusion
follows.
\end{proof}
$\mathbf{Remark}$ In \cite[Theorem 4.1]{501}, Alaminos et al.  show that if $\phi$ is a bounded linear mapping  from a C*-algebra $\mathcal{A}$
into an essential Banach $\mathcal{A}$-bimodule $\mathcal{M}$ satisfying the condition $(\mathbb{P}),$  then $\phi=\delta+\eta,$ where $\delta$ is a derivation and $\eta$ is a multiplier.
Theorem 2.16 tells that for the case of von Neumann algebra, the condition of boundness is not necessary.

\emph{Acknowledgements}. This paper was partially supported by National Natural Science Foundation of China(Grant No. 11371136).


\begin{thebibliography}{99}


\bibitem{521} J. Alaminos, M. Bre\v{s}ar, J. Extremera,  A. Villena, \textit{
Characterizing homomorphisms and derivations on
$C^{*}$-algebras,}  Proc. Roy. Soc. Edinb.  137A,(2007), 1-7.


\bibitem{501} J. Alaminos, M. Bre\v{s}ar, J. Extermera,  R. Villena, \textit{Characterizing Jordan maps on $C^{*}$-algebras through zero products,}
~Proc. Edinburgh Math. Soc.  53(2010),  543-555.

\bibitem{523} J. Alaminos, M. Bre\v{s}ar, J. Extremera,  A. Villena,  \textit{Maps preserving zero products,}
Studia Math. 193(2009),  131-159.

\bibitem{448} R. Alizadeh,   \textit{Jordan derivations of full matrix algebras,}
 Linear  Algebra  Appl.  430(2009),  574-578.


\bibitem{448}  G. An, J. Li, J. He, \textit{Jordan product determined by zero products,}  Linear  Algebra  Appl. 475(2015), 90-93.

\bibitem{446}  G. An, J. Li,  \textit{Characterizations of linear mappings through zero products or zero Jordan products,}  Electron. J. Linear Algebra. 31(2016), 408-424.

\bibitem{999} R. An, J. Hou, \textit{Characterizations of Jordan derivations on rings with idempotent,}  Linear Multilinear Algebra. 58(2010), 753-763.

\bibitem{M. Bresar J. Vukman1} M. Bre\v{s}ar,  J. Vukman,  \textit{Jordan derivations on prime rings},
Bull. Austral. Math. Soc. 37(1988),  321-322.


\bibitem{M. Bresar2} M. Bre\v{s}ar, \textit{Jordan derivations on semiprime rings},
Bull. Austral. Math. Soc. 104(1988),  1003-1006.


\bibitem{Cusack} J. Cusack, \textit{Jordan derivations on rings},
Proc. Amer. Math. Soc. 53(1975),  321-324.


\bibitem{H. Dales} H. Dales, \textit{Banach algebras and automatically continuous},
London Mathematical Society Monographs 24, Clarendon Press,  Oxford, 2000.


\bibitem{587} K. Davidson,  \textit{Nest algebras}, Pitman Research Notes in Math Series 191(1988).


\bibitem{Ahlem Ben Ali Essalen} A. Essalen, A. Peralta, M. Ramirez,
\textit{Weak-local derivations and homomorphisms on
$C^{*}$-algebras}, Linear Multilinear Algebra.
64(2016), 169-186.


\bibitem{517} F. Gilfeather, R. Moore,
\textit{Isomorphisms of certain CSL algebras}, J. Func. Anal. 67(1986), 264-291.


\bibitem{S. Goldstein} S. Goldstein,  A. Paszkiewicz,  \textit{Linear combinations of projections in von Neumann algebras},
Proc. Amer. Math. Soc.  116(1992),  175-183.

\bibitem{888} D. Hadwin,  J. Li,  \textit{Local derivations and local automorphisms,}  J. Math. Anal. Appl. 290(2003), 702-714.

\bibitem{} D. Hadwin,  J. Li,  \textit{Local derivations and local automorphisms on some algebras,}  J. Operator Theory 60(2008), 29-44.


\bibitem{I. Herstein} I. Herstein, \textit{Jordan derivations of prime rings},
Proc. Amer. Math. Soc. 8,(1957),  1104-1110.

\bibitem{R.Kadison J. Ringrose} R. Kadison,  J. Ringrose,  \textit{Fundamentals of the Theory of Operator Algebras},
Academic Press Inc. 1983.


\bibitem{599} D. Liu, J. Zhang,
\textit{Jordan derivable maps on triangular algebras by commutative zero products,} Acta Math. Sin. Chin. Ser. 57(2014), 1203-1208.


\bibitem{518} F. Lu,
\textit{Lie derivations of certain CSL algebras}, Israel. J. Math. 155(2006), 149-156.


\bibitem{580} J. Li,  Z. Pan,  \textit{Annihilator-preserving maps, multipliers, and derivations,}
 Linear  Algebra Appl.  432(2010),  5-13.

\bibitem{222} A. Peralta,  B. Russo,  \textit{Automatic continuity of derivations on C*-algebras and JB*-triples,}
 Journal of Algebra.  399(2014),  960-977.


\bibitem{Ringrose} J. Ringrose, \textit{Automatically continuous of derivations of operator algebras},
J. London Math. Soc.  5(1972),  432-438.



\bibitem{sakai} S. Sakai,  \textit{$C^{*}$-algebras and  $W ^{*}$- algebras},
Springer-Verlag. 1971.

\bibitem{444} J. Zhang,  W. Yu,  \textit{Jordan derivations of triangular algebras,}
 Linear  Algebra Appl.  419(2006),  251-255.

\end{thebibliography}
\end{document}